\documentclass[11pt]{amsart}
         \usepackage[paperwidth=7in, paperheight=10in, margin=.875in]{geometry}
 \usepackage[backref,colorlinks,linkcolor=red,anchorcolor=green,citecolor=blue]{hyperref}
\usepackage{amsfonts,amssymb}
\usepackage{amsmath}
\usepackage{mathrsfs} 
\usepackage{graphicx}
\graphicspath{ {./images/} }
\usepackage{cite}
\usepackage{enumerate}
\usepackage{tikz}
\usepackage{caption}
\usepackage{easyReview}
\numberwithin{equation}{section}
\newtheorem{thm}{Theorem}[section]

\newtheorem{prop}[thm]{Proposition}
\newtheorem{cor}[thm]{Corollary}

\theoremstyle{remark}

\theoremstyle{definition}

\newcommand{\Gnorm}[1]{\left\lVert#1\right\rVert}

\newcommand{\gbinom}[2]{\genfrac{[}{]}{0pt}{}{#1}{#2}}

\newtheorem{rem}[thm]{Remark}

\DeclareMathOperator*{\argmax}{argmax}
\makeatletter
\@namedef{subjclaame@2020}{\textup{2020}}
\makeatother

   \allowdisplaybreaks
\begin{document}
 \title[]{Approximating the first passage time density from data using generalized Laguerre polynomials}
\author[E. Di Nardo]{Elvira Di Nardo$^{\ast}$}
\address{$^{\ast}$ Dipartimento di Matematica \lq\lq G. Peano\rq\rq, Università degli Studi di Torino, Via Carlo Alberto 10, 10123 Torino, Italy}
\email{elvira.dinardo@unito.it }
\author[G. D'Onofrio]{Giuseppe D'Onofrio$^{\dagger}$}
\address{$^{\dagger}$ Dipartimento di Matematica \lq\lq G. Peano\rq\rq, Università degli Studi di Torino, Via Carlo Alberto 10, 10123 Torino, Italy}
\email{giuseppe.donofrio@unito.it}

\author[T. Martini]{Tommaso Martini$^{\ast\ast}$}
\address{$^{\ast\ast}$ Dipartimento di Matematica \lq\lq G. Peano\rq\rq, Università degli Studi di Torino, Via Carlo Alberto 10, 10123 Torino, Italy}
\email{tommaso.martini@unito.it }        

         \pagestyle{myheadings}  \maketitle

\begin{abstract}
This paper analyzes a method to approximate the first passage time probability density function which turns to be particularly useful if only sample data are available. The method relies on a Laguerre-Gamma polynomial approximation and iteratively looks for the best degree of the polynomial such that the fitting function is a probability density function. The proposed iterative algorithm relies on simple and new recursion formulae involving first passage time moments. These moments can be computed recursively from cumulants, if they are known. In such a case, the approximated density can be used also for the maximum likelihood estimates of the parameters of the underlying stochastic process.  If cumulants are not known, suitable unbiased estimators relying on $\kappa$-statistics are employed.
To check the feasibility of the method both in fitting the density and in estimating the parameters, the first passage time problem of a geometric Brownian motion is considered.  
\end{abstract}
\noindent {\small {\bf keywords: }{stochastic process, cumulant, geometric Brownian motion, $\kappa$-statistic, recursive algorithm}} 

\section{\label{sec:level1}
Introduction }

The first-passage-time (FPT) problem arises in
many applications in which a stochastic process $Y(t)$ starting in $y_{\tau}$ at time $\tau$ evolves in the presence of a  threshold $S(t).$ They range from finance to engineering including, among others, computational neuroscience, mathematical biology and reliability theory, see   \cite{redner} for a comprehensive collection of results.
The mathematical study of the FPT problem consists in finding the probability density function (pdf) 
$g[S(t),t|y_{\tau},\tau]=\frac{d}{d t}\mathbb{P}\{T<t\}$ of the random variable $T$, defined by
\begin{equation}\label{12}
T=\begin{cases} \inf_{t\geq \tau}\{Y(t)>S(t)\},& Y(\tau)=y_{\tau}<S(\tau),\\
\inf_{t\geq \tau}\{Y(t)<S(t)\},& Y(\tau)=y_{\tau}>S(\tau).\\
\end{cases}
\end{equation}
There are several strategies to approach this problem,
whose effectiveness depends on the properties of the stochastic process involved (see \cite{MR1718867} for an extensive review). 
They range from 
the Doob's representation formula \cite{doob1949}
to the Siegert's equation \cite{Siegert1951} that consists in a partial differential equation involving either the moments of $T$ or its Laplace transform. 
Closed-form expressions emerge only in a few cases \cite{buonocore2015closed} but numerical evaluation of $g$ can be provided as solution of non-singular second kind Volterra integral equation 
\cite{buo_integral, ricciardi_sacerdote_sato_1984, giorno_nobile_ricciardi_sato_1989, torres, dinardo_nobile_pirozzi_ricciardi_2001} or using Sturm-Liouville eigenfunction expansion series
\cite{Linetsky, alili_patie, Kent1980}. Due to the difficulty of the problem, asymptotic expressions of $g$ are studied using the Volterra integral equation \cite{giorno1990, nobile_ricciardi_sacerdote_1986}, Laplace transform techniques \cite{Martin_2019}
or Large Deviation estimates \cite{baldi_car_rossi, macci}.

All these techniques rely on the knowledge of the nature of the stochastic process $Y(t).$ But if a random sample of FPTs is analyzed without any prior information on the stochastic dynamics generating the data, the identification of a model could be difficult to implement. In general classical tools as histograms or kernel density estimators are the first choices to make a guess on the shape of the FPT pdf and then postulate a model.

Following the latter approach, the aim of this paper is to propose a Laguerre-Gamma polynomial approximation to fit a suitable pdf on a random sample of FPTs. This proposal moves from some recent approaches \cite{Ramos, MR4159245, mathematics2021} that use cumulants to obtain FPT pdf and related statistics.

Recall that if $T$ has moment generating function  $\mathbb E[e^{zT}] < \infty$ for all $z$ in an open interval about $0,$ then its cumulants $\{c_k(T)\}_{k \geq 1}$ are such that 
\begin{equation*}
\sum_{k \geq 1} c_k(T) 
\frac{z^k}{k!}  = \log \mathbb E[e^{zT}] 
\label{defcum}
\end{equation*}
for all $z$ in some (possibly smaller) open interval about $0.$ It is possible to recover cumulants up to order $k$ from
moments up to the same order (and viceversa), using the general partition polynomials \cite{Charalambides}. Therefore, from a theoretical point of view, there is a duality between these two numerical sequences, even if 
 the expressions of cumulants are most of the time simpler than those of moments. Moreover, 
cumulants have nice properties compared with moments such as the semi-invariance and the additivity. Overdispersion and underdispersion as well as asymmetry and tailedness of the FPT pdf might be analized through the first few cumulants \cite{McCullagh}.

The approximation proposed in this paper has a twofold
advantage. If the FPT moments/cumulants are not known, the special feature of this approach is the chance to recover an approximation of
the FPT pdf starting from a sample of FPT data like the classical
density estimators. Indeed estimates carried out from $\kappa$-statistics can replace the occurrences of cumulants in the polynomial approximation. Let us recall that the $k$-th $\kappa$-statistic is a symmetric function of the random sample whose expectation gives the $k$-th cumulant
$c_k(T).$ These estimators have minimum variance when compared to all other unbiased estimators and are built by free-distribution methods without using sample moments
\cite{Kendall}.
The method turns to be useful also if the model is known but the knowledge of the FPT moments is limited, as usually happens. In such a case, the approximation might be carried out  by simulating  the trajectories of the process  through a suitable  Monte Carlo method and using $\kappa$-statistics in place of cumulants. If the FPT moments/ cumulants are known or can be recovered from the Laplace transform of the FPT random variable $T$, the method is essentially a way to find an approximated analytical expression of $g.$ The approach works for a wide class of one dimensional stochastic processes and,  of course, is intended for the cases in which the closed form expression of $g$ is not available. Differently from other methods, 
the approximating function results to be a pdf whatever order of approximation is reached. Therefore it is possible to implement a maximum likelihood procedure to carry out estimates of the parameters involved in the model.

To check the feasibility of our proposal we compare the approximated expressions (obtained analytically or  estimated through $\kappa$-statistics) with a case in which the FPT pdf is known. In particular we use the closed analytical expression of the pdf of a Geometric Brownian motion (GBM) FPT through a constant boundary $S$. 

The GBM is a special case from the point of view of the FPT problem.
In this case the Laplace transform of $g$ is known and can be inverted, obtaining a closed form expression for the FPT density.
Moreover one can use the Doob's transform since the GBM can be generated from a Brownian motion with drift (case for which the pdf of $T$ is known), it has stationary distribution and its parameters solve the Siegert's equations. Moreover a Sturm-Liouville eigenvalue expansion for its infinitesimal generator exists \cite{byczkowski2006GBM}, its $g$ can be found directly from the Volterra integral equation \cite{MR1464598}. 
To test the performance of the approximated analytical expression here proposed, the choice of the GBM allows us to overcome the numerical difficulties that can arise performing simulations of the FPT, when exact method are not implemented \cite{herrmann2020exact}. In fact, in this case, the proposed polynomial approximation can be compared directly with a closed-form expression.

The paper is organized as follows.
In Section \ref{section2} we resume the FPT problem for the GBM recalling the known results useful for carrying out the approximation including some  new expressions for moments and cumulants relied on exponential Bell polynomial.
In Section \ref{section3} we discuss sufficient and necessary conditions to recover the infinite series expansion of the GBM FPT pdf in terms of the generalized Laguerre polynomials  and study the convergence of the method when this series is truncated at the order $n$.
In particular we investigate the chance to consider reference pdfs different from the usual Gamma density. 
Section \ref {section4} collects results and tools for an efficient iterative search of the best degree of approximation $n$. 
Numerical results on Laguerre-Gamma approximation are given in Section \ref{section5}. In particular we estimate simultaneously two parameters of the GBM using  the maximum likelihood estimation starting from first passage time data.
The method relies on the Laguerre-Gamma approximated FPT pdf and can be used even if $g$ is unknown. In this case it relies on $k$-statistics.

\section{The FPT problem for the Geometric Brownian Motion} \label{section2}
The GBM  is a regular diffusion process on $(0,+\infty)$.
It can be obtained as a transformation of a Brownian motion and for this reason it is also called exponential Brownian motion \cite{yor}. In fact, if 
$X(t)=(\mu-\frac{\sigma^2}{2})t+\sigma\; W(t)$ is a drifted Brownian motion, then the stochastic process
\begin{equation}\label{22}
Y(t)=y_0 \;e^{(\mu-\frac{\sigma^2}{2})t+\sigma W(t)}=y_0 \;e^{X(t)} \qquad \hbox{ with $\mu>\frac{\sigma^2}{2}$}
\end{equation}
is said GBM with starting point $y_0$ and
with infinitesimal mean and variance $(\mu-\frac{\sigma^2}{2})y$ and $\sigma^2y^2$, respectively \cite{karlin1981second}.

The transition pdf is known to be a lognormal pdf with parameters $\mu-\frac{\sigma^2}{2}$ and $\sigma\sqrt{t}$,
obtained by solving the Fokker-Planck equation, which in the case of the GBM reduces to the canonical form of the heat equation.
The transition pdf  presents exponential decay in the variable $t$ so the Laplace transform exists for every $\lambda >0$.

Let us assume that the threshold $S$ is constant
and $y_0 < S.$ In this case the Laplace transform $g_{\lambda}(S|y_0)$ of $g$ is given by \cite{MR1912205}
$$
	g_{\lambda}(S|y_0)=\bigg(\frac{y_0}{S}\bigg)^{k(\lambda)}  \quad \text{with } \quad 
k(\lambda)=	\frac{1}{\sigma^2}\left[ \sqrt{\bigg(\mu-\frac{\sigma^2}{2}\bigg)^2 + 2 \sigma^2 \lambda} +\bigg(\frac{\sigma^2}{2} -\mu\bigg)\right]$$
which can be rewritten as
\begin{equation}
 \label{28bis}
	g_{\lambda}(S|y_0)= \exp\left\{ k(\lambda) \ln \bigg(\frac{y_0}{S}
	\bigg)\right\} = 
	\exp\left\{\frac{a}{b}
	\bigg( 1 - \sqrt{1+\frac{2 b^2}{a} \lambda}\bigg)\right\}
\end{equation}
with 
\begin{equation}
a = \frac{(\ln S - \ln y_0)^2}{\sigma^2} > 0 \quad \text{and } \quad 
b = \frac{\ln S - \ln y_0}{\mu - \frac{\sigma^2}{2}} > 0.
\end{equation}
Since the moment generating function $M_T(\lambda) = {\mathbb E}[\exp(T \lambda)]$ is such that  $M_T(\lambda)= g_{-\lambda}(S|y_0),$ from \eqref{28bis}
$T$ has inverse gaussian distribution  $IG(a,b)$ of parameters $a$ (the shape) and $b$ (the mean), 
that is
\begin{equation}
\label{FPTpdf}
	g(t)=\sqrt{	\frac{a}{2\pi t^{3}}}\;
	\exp\bigg(-\frac{a(t-b)^2}{2 b^2 t}\bigg), \quad t > 0
	\end{equation}
where $g(t)$ denotes the FPT pdf $g(S,t|y_0)$ from now on. Statistical properties of $IG(a,b)$ have been investigated in \cite{MR110132}. Thus moments of $T$ result to be
\begin{equation}
\label{2.5}
{\mathbb E}[T^n] = b^n \sum_{k=0}^{n-1} \frac{(n-1+k)!}{k!(n-1-k)!} \frac{b^k}{(2 a)^k} =  \frac{\exp(a/b)}{b^{\frac{1}{2}-n}} {\mathcal K}_{n-\frac{1}{2}}\bigg(\frac{a}{b}\bigg)\sqrt{\frac{2 a}{\pi}}
\end{equation}
where ${\mathcal K}_{\nu}(z)$ is the modified Bessel function of second type \cite{MR1243179}
\begin{equation}
\label{modfbessel}
{\mathcal K}_{\pm \nu}(z)
= \frac{1}{2} \left( \frac{z}{2} \right)^{\nu}
\int_0^{\infty} x^{-\nu-1}
\exp \left( -x - \frac{z^2}{4 x}\right) \, {\rm d}\, x
\end{equation}
under the conditions $|{\rm arg}\, z|<\pi/2$ and ${\rm Re}\, z^2 > 0.$ Since
$z {\mathcal K}_{ \nu-1}(z) - z{\mathcal K}_{\nu+1}(z)=-2\nu {\mathcal K}_{\nu}(z)$\cite{MR1243179},
the following recursion formula holds for the FPT moments:
\begin{equation}
{\mathbb E}[T^{n+1}] = \frac{(2n-1)b^2}{a}  
{\mathbb E}[T^{n}] + b^2 {\mathbb E}[T^{n-1}], \quad {n \geq 1}
\label{recursionmom}
\end{equation}
with ${\mathbb E}[T^{0}]=1$ and ${\mathbb E}[T]=b.$ 

A new and alternative expression of the FPT moments can be given using
the partition polynomial
\cite{Charalambides}
\begin{equation}
G_n(y;x_1,\ldots,x_n)=\sum_{j=1} y^j B_{n,j}(x_1,  \ldots, x_{n-j+1})  
\end{equation}
where $\{B_{n,j}\}$
are the partial exponential Bell polynomials   
\begin{equation}\label{(parexpBell)}
B_{n,j}(x_1,  \ldots, x_{n-j+1}) =  \sum \frac{n!}{i_1! i_2! \cdots  i_{n-j+1}!} \prod_{k=1}^{n-j+1} \left(\frac{x_{k}}{k!}\right)^{i_k}
\end{equation}
with the sum taken over all sequences $i_1, i_2, \ldots, i_{n-j+1}$ of non negative integers  such that $i_1 + 2 i_2  + \cdots + (n-j+1) i_{n-j+1}= n$ and $i_1 + i_2 + \cdots + i_{n-j+1}  = j.$

We are now ready to  give the new (polynomial) closed-form expression of the moments of $T$ that relies on the simple expression of the cumulants,
overcoming the evaluation of the ${\mathcal K}_{\nu}(z)$ in Eq.\eqref{2.5}.
\begin{prop} 
\begin{equation}
{\mathbb E}[T^n] = \left( -\frac{2 b^2}{a} \right)^n G_n \left(-\frac{a}{b}; \frac{1}{2}, \left(\frac{1}{2}\right)_2, \ldots,  \left(\frac{1}{2}\right)_{n-j+1}\right) 
\label{cf1}
\end{equation}
\end{prop}
\begin{proof}
From the power series expansion of  $\ln [M_T(\lambda) - 1]$
with $M_T(\lambda)$ the moment generating function of $T$, cumulants  $\{c_n[T]\}$ of $T$ result to be \cite{MR110132}
\begin{equation}
\label{cumulants}
c_n[T] = (2 n - 3)!! \frac{b^{2n-1}}{a^{n-1}} \quad \text{for } n \geq 1
\end{equation}
with $(2 n - 3)!! = (2n-3) \cdots 5 \cdot 3 \cdot 1.$
Since $(2 n - 3)!! = (-1)^{n-1} 2^n \left( \frac{1}{2} \right)_n$ 
with 
$\left( \frac{1}{2} \right)_n = \prod_{j=0}^{n-1} (\frac{1}{2} - j)$ the lowering factorial, 
 from \eqref{cumulants}
we also have 
 \begin{equation}
\label{cumulants2}
c_n[T] =  \left(-\frac{a}{b}\right) \left(-\frac{2 b^2}{a}\right)^n  \left( \frac{1}{2} \right)_n  \quad \text{for } n \geq 1.
\end{equation}
Moments ${\mathbb E}[T^n]$
are related to cumulants
through the partial exponential Bell  polynomials \cite{DiNardo}  \begin{equation}
{\mathbb E}[T^n] = 
\sum_{j=1}^n B_{n,j}(c_1[T], \ldots, c_{n-j+1}[T]), \qquad \hbox{\rm for $n \geq 1,$}
\label{(comexpBell)}
\end{equation}
with $\{B_{n,j}\}$ given in \eqref{(parexpBell)}. 
The result follows replacing \eqref{cumulants2} in 
\eqref{(comexpBell)} and
using the well known property $B_{n,j}(p q x_1, p q^2 x_2, \ldots) = p^j q^n B_{n,j}( x_1,  x_2, \ldots)$  \cite{MR0262087}.
\end{proof}
Let us observe that from \eqref{cumulants}, the following recursion holds  for the FPT cumulants 
$$c_n[T]=  \frac{(2 n - 3) b^2}{a} c_{n-1}[T]$$
starting with $c_{1}[T] = {\mathbb E}[T] = b.$

\section{Series expansion}\label{section3}

To approximate
a pdf over $(0, \infty)$ using moments, a classical method \cite{MR388736} consists in expanding this function as an infinite series involving the generalized Laguerre polynomials $\{L_k^{(\alpha)}(t)\}_{k \geq 0}$ defined
for $\alpha > -1$ by 
$$L_0^{(\alpha)}(t) = 1 \quad
\text{and } \quad L_k^{(\alpha)}(t) = \sum_{i=0}^k 
\binom{k+\alpha}{k-i} \frac{(-t)^i}{i!}, \,\, k \geq 1.$$
No sufficient conditions are discussed in \cite{MR388736} to ensure that the pdf can be expanded formally as
an infinite series of generalized Laguerre polynomials, because the pdf is not known in general. This method has been applied to approximate the FPT pdf of a square-root process in  \cite{MR4159245}
giving also some sufficient conditions to justify the approximation.

Here we take advantage of the knowdlege of the FPT pdf $g(t)$ of a GBM  \eqref{FPTpdf} to discuss some conditions in order to use this method.  A first step in this direction is to recall the classical results on the completeness and orthogonality of $\{L_k^{(\alpha)}(t)\}$  in the weighted Hilbert space ${\mathcal L}^2_{t^{\alpha} e^{-t}}(0,\infty)$
equipped with the  inner product $\langle f_1,f_2 \rangle =\int_0^{\infty} f_1(t) \, f_2(t) \, t^{\alpha} \, e^{-t} \, {\rm d} t,$ see for example  \cite{MR1118381}. Using these properties, the following proposition provides the infinite series expansion of $g(t)$
in terms of generalized Laguerre polynomials. 
\begin{prop} \label{conv1} If \begin{equation}\label{suffcond}
I= \int_0^{\infty} t^{-(\alpha + 3)} \exp\bigg( - A t - \frac{a}{t} \bigg) {\rm d}\, t < \infty,
\end{equation}
with $A = \frac{a}{b^2} - \beta$  then for $t > 0$  
\begin{equation}
\label{expansion1}
g(t) = f_{\alpha,\beta}(t)
\bigg( 1 + \sum_{k \geq 1} {\mathcal B}_k^{(\alpha)} L_k^{(\alpha)}(\beta t)\bigg) \,\, \text{with } \,\, {\mathcal B}_k^{(\alpha)} = 1 + \sum_{j=1}^k \binom{k}{j} \frac{(-\beta)^j \mathbb E[T^j]}{(\alpha+j)_j},
\end{equation}
where $f_{\alpha,\beta}(t) = \beta (\beta t)^{\alpha} e^{-\beta t} / \Gamma(\alpha+1)$ is the gamma pdf with scale parameter $\alpha+1>0$ and shape parameter $\beta>0.$ 
\end{prop}
\begin{proof}
Consider the weighted Hilbert space ${\mathcal L}^2_{f_{\alpha,\beta}(t)}(0,\infty)$ equipped with the inner product $\langle f_1,f_2 \rangle =\int_0^{\infty} f_1(t) \, f_2(t) \, f_{\alpha,\beta}(t)  \, {\rm d} t$ and set 
$$a_k^{(\alpha)} = (-1)^k \bigg(\frac{\Gamma(\alpha+1+k)}{k! \,\Gamma(\alpha+1)}\bigg)^{-1/2} \,\, \text{for }k \geq 1.$$
By recalling that
$\int_0^{\infty} t^{\alpha} \exp(-t) L_n^{(\alpha)}(t) 
L_m^{(\alpha)}(t)  \, {\rm d} t = \Gamma(n + \alpha + 1)/n! \delta_{n,m},$ with $\delta_{n,m}$ the Kronecker
delta function, the sequence $\{a_k^{(\alpha)} L^{(\alpha)}_k(\beta t)\}$
turns to be complete and orthonormal in ${\mathcal L}^2_{f_{\alpha,\beta}(t)}(0,\infty).$ Condition \eqref{suffcond} is equivalent to require
$g(t)/f_{\alpha,\beta}(t) \in {\mathcal L}^2_{f_{\alpha,\beta}(t)}(0,\infty)$ since 
\begin{equation}
\label{suffcond1}
\int_0^{\infty}
\frac{[g(t)]^2}{ f_{\alpha,\beta}(t)} \, {\rm d}\, t = \frac{a \, \Gamma(\alpha+1) \, e^{2 a/b}}{ 2 \pi \beta^{\alpha+1}} \, I.
\end{equation}
Due to the completeness of the sequence, the pdf $g(t)$ can be expanded in terms of 
$\{a_k^{(\alpha)} L^{(\alpha)}_k(\beta t)\},$ that is
\begin{equation}
\label{expansion0}
\frac{g(t)}{f_{\alpha,\beta}(t)} = \sum_{k \geq 0} \big[a_k^{(\alpha)}\big]^2
\mathbb E\big[L^{(\alpha)}_k(\beta T)\big]L^{(\alpha)}_k(\beta t).
\end{equation}
Expansion \eqref{expansion1} follows after some algebraic manipulations of the rhs of \eqref{expansion0}.
\end{proof}
\begin{cor}
Condition \eqref{suffcond} is fulfilled iff 
\begin{equation}
    \beta \leq \frac{c_1[T]}{
    c_2[T]} = \frac{1}{\sigma^2} \left( \mu - \frac{\sigma^2}{2} \right)^2,
    \label{condcum}
\end{equation}
with $c_1[T]$ and $c_2[T]$ the first and the second cumulant, respectively.
\end{cor}
\begin{proof}
Note that \eqref{condcum} is equivalent to have
$A = \frac{1}{\sigma^2} \left( \mu - \frac{\sigma^2}{2} \right)^2 - \beta \geq 0$
in \eqref{suffcond}. If $A > 0$ from $3.471$ no. $9$  in \cite{MR1243179}, we get
$$ \int_0^{\infty} t^{-(\alpha + 3)} \exp\bigg( - A t - \frac{a}{t} \bigg) {\rm d}\, t = 2 \left( \frac{A}{a} \right) ^{\frac{2+\alpha}{2}} {\mathcal K}_{\alpha+2}(2 \sqrt{a A}) < \infty 
$$
where ${\mathcal K}_{\nu}(z)$ is the modified Bessel function of second type. If $A=0,$ by a suitable change of variable and using $3.478$ no. $1$ in \cite{MR1243179}, we have
$$\int_0^{\infty} t^{-(\alpha + 3)} \exp\bigg( - A t - \frac{a}{t} \bigg) {\rm d}\, t = \int_0^{\infty} x^{(\alpha+2)-1} e^{-a x}
{\rm d} x = a^{-(\alpha+2)} \Gamma(\alpha+2) < \infty.$$
If $A<0,$ the integrand function in $I$
grows with $t$ and the condition 
\eqref{suffcond} is not fulfilled.
\end{proof}
Propositions \ref{conv1} justifies the approximation of the FPT pdf $g(t)$ with $\hat{g}_n(t)$ 
\begin{equation}
\label{expansion2}
\hat{g}_n(t) = f_{\alpha,\beta}(t) p_n(t)
\quad \text{with }\quad
p_n(t)=\bigg( 1 + \sum_{k = 1}^n {\mathcal B}_k^{(\alpha)} L_k^{(\alpha)}(\beta t)\bigg)
\end{equation}
a polynomial of degree $n \geq 1$ for a suitable choice of $n.$
Due to the orthogonality property of generalized Laguerre polynomials, we observe that $\int_0^{\infty}  \hat{g}_n(t) {\rm d}t = 1$
for all $n \geq 0,$  and
the first $n$ moments of
$\hat{g}_n(t)$ are the same of $g(t).$ 

The main issue in the approximation \eqref{expansion2} is the choice of the degree $n$ of the polynomial $p_n(t),$ that we discuss numerically in the next sections. The higher is the order $n$ the better should be the approximation. Indeed 
using the Parseval's formula \cite{MR1563024}
the error in replacing $g(t)$ with $\hat{g}_n(t)$ for $t>0$
is such that
\begin{equation}
\label{errorPars}
\Gnorm{ \frac{g(t)}{f_{\alpha,\beta}(t)} - p_n(t)}_{\alpha,\beta} = o \left( \frac{1}{n} \right) \quad 
\text{as } n \rightarrow \infty
\end{equation}
where $\Gnorm{\,\,}_{\alpha,\beta}$ denotes 
the norm in ${\mathcal L}^2_{f_{\alpha,\beta}(t)}(0,\infty).$ 
The following theorem improves \eqref{errorPars} for suitable choice of $k <n.$
\begin{thm} \label{3.3}
If $\beta < c_1[T]/c_2[T]$ and $k \in {\mathbb N},$ then there exists a constant $C_k > 0$  such that
\begin{equation}
\label{error1}
\Gnorm{ \frac{g(t)}{f_{\alpha,\beta}(t)} - p_n(t)}_{\alpha,\beta} \leq C_k \bigg( \frac{1}{\sqrt{n}}
\bigg)^k \quad \hbox{\rm for all } \,\, n > k
\end{equation}
with $C_k = O(k^k).$
\end{thm}
\begin{proof}
Set $\tilde{a}=\frac{a}{2}$ and $\tilde{A} = \frac{a}{2b^2} - \beta$ in order to write
$$\frac{g(t)}{f_{\alpha,\beta}(t)} = D \,
\exp \left( - \tilde{A} t -  \frac{\tilde{a}}{t} \right) \, t^{-\frac{3}{2} - \alpha} \quad \text{with } \quad D = \sqrt{\frac{\tilde{a}}{ \pi}} \frac{e^{a/b} \Gamma(\alpha+1)}{\beta^{\alpha+1}}.$$ 
According to Theorem 6.2.5 in \cite{MR1176949}, if for a fixed $k \in {\mathbb N}$ 
\begin{equation}
\label{cond1bis}
t^{m/2} \frac{{\rm d}^m}{{\rm d}t^m} \left[ \frac{g(t)}{f_{\alpha,\beta}(t)}\right] \in {\mathcal L}^2_{f_{\alpha,\beta}(t)}(0,\infty) \quad \hbox{for } \quad 0 \leq m\leq k
\end{equation}
then there exists a constant $C > 0$  such that
\begin{equation}
\label{book}
\Gnorm{ \frac{g(t)}{f_{\alpha,\beta}(t)} - p_n(t)}_{\alpha,\beta} \leq C \bigg( \frac{1}{\sqrt{n}}
\bigg)^k \Gnorm{ t^{k/2} \frac{{\rm d}^k}{{\rm d}t^k} \left[ \frac{g(t)}{f_{\alpha,\beta}(t)}\right]}_{\alpha,\beta}
\quad \hbox{\rm for all} \,\,\, n > k.
\end{equation}
By recursion, we have
\begin{equation}
\frac{{\rm d}^m}{{\rm d}t^m} \left[ \frac{g(t)}{f_{\alpha,\beta}(t)} \right] = \frac{(-1)^m}{2^m} D \, 
\exp \left( - \tilde{A} t -  \frac{\tilde{a}}{t} \right) \, t^{-2 m - \frac{3}{2} - \alpha} \, q_{2m}(t) \quad \hbox{for $m \geq 0$ }
\label{derivf}
\end{equation}
where $q_{2m}(t)$ is a polynomial of degree $2m$
 such that
\begin{equation}
q_{2m}(t) = \sum_{j=0}^{2k}c_{2m,j} t^j = q_{2m-2}(t)[2 \, \tilde{A} \, t^2 + (4m-1+2\alpha)\,t-2\,\tilde{a}]-\,2t^2 \, \frac{{\rm d}}{{\rm d}t}  q_{2m-2}(t)
\label{polrec}
\end{equation}
with $q_0(t)=1$ and $c_{2m,2m} = (2 \tilde{A})^m \ne 0, c_{2m,0} =(-2 \tilde{a})^m.$

Now set $$\tilde{q}_{4m}(t)=[q_{2m}(t)]^2 = \sum_{j=0}^{4m} \tilde{c}_{4m,j}t^j 
\,\,\,
\text{ and } \,\,
I_{m}= \int_0^{\infty}
t^{m} \left( \frac{{\rm d}^m}{{\rm d}t^m} \left[ \frac{g(t)}{f_{\alpha,\beta}(t)} \right] \right)^2 f_{\alpha,\beta}(t) {\rm d}t.$$
Fix an integer $0 \leq k <n.$ Since  $2 \tilde{A} + \beta = \frac{c_1[T]}{
c_2[T]} - \beta > 0$ we have $I_m < \infty$ for all $m \geq 0$ and in particular
condition \eqref{cond1bis} holds for $0 \leq m \leq k.$  
Indeed using 
the modified Bessel function of second type \eqref{modfbessel} and the integral $3.471$ no. $9$  in \cite{MR1243179}, we have
\begin{eqnarray}
&  & \qquad I_{m} =
\frac{D^{2}}{2^{2m}} \sum_{j=0}^{4m} \tilde{c}_{4m,j} \int_0^{\infty} \exp \left( - (2 \tilde{A} + \beta) t -   \frac{a}{t} \right) \, t^{(-3m-2-\alpha+j) -1} \,
 {\rm d} \,t  \label{IM} \\
 &  & =  \frac{D^{2}}{2^{2m-1}} \sum_{j=0}^{4m} \tilde{c}_{4m,j} {\mathcal K}_{3m-j+(\alpha+2)}\bigg(\sqrt{4 a(2 \tilde{A} + \beta)}\bigg)
 \left[ \sqrt{\frac{a}{2 \tilde{A} + \beta}} \right]^{-3m -2-\alpha+j} < \infty.
 \nonumber
\end{eqnarray}
Eq. \eqref{error1} follows from \eqref{book} setting  $m=k$ and 
\begin{equation}
\label{Ck}
C_k = C I_k \propto \frac{1}{2^{2k}} \sum_{j=0}^{4k} \tilde{c}_{4k,j} {\mathcal K}_{3k-j+(\alpha+2)}\bigg(\sqrt{4 a(2 \tilde{A} + \beta)}\bigg)
 \left[ \sqrt{\frac{a}{2 \tilde{A} + \beta}} \right]^{-3k -2-\alpha+j}.  
\end{equation}
In \eqref{Ck}, note that $3k-j+\alpha+2 > 0$ for $j=0, \ldots, 3k +1$ as $\alpha+1>0.$ For $3k+2 \leq j \leq 4k,$ the order of the modified Bessel function involved in $C_k$ might be positive, depending on the magnitude of $\alpha.$ Let us first suppose $\alpha > k-2$
such that $3k-j+\alpha+2 > 0$ for all $j =0, \ldots, 4k.$
As for $\nu \rightarrow \infty$ one has \cite{bessel}
\begin{equation}
\label{asint}
{\mathcal K}_{\nu}(z) \sim \sqrt{\frac{\pi}{2 \nu}}
\left( \frac{2 \nu}{ e z} \right)^{\nu}
\end{equation}
then
\begin{equation}
\label{firstas}
C_k  \propto \frac{1}{2^{2k}} \sum_{j=0}^{4k} \frac{\tilde{c}_{4k,j}}{\sqrt{3k-j+\alpha+2}} \left( \frac{3k-j+\alpha+2}{e a} \right)^{3k-j+\alpha+2}.
\end{equation}
When $k$ grows, the dominant term in \eqref{firstas} is for $j=0,$ and the result follows.
If $\alpha < k-2,$
then $C_k$ might be splitted in 
$C_k \propto C_{k,1} + C_{k,2}$ with
\begin{eqnarray*}
C_{k,1} & = & \frac{1}{2^{2k}} \sum_{j=0}^{k^*} \tilde{c}_{4k,j} {\mathcal K}_{3k-j+(\alpha+2)}\bigg(\sqrt{4 a(2 \tilde{A} + \beta)}\bigg)
 \left[ \sqrt{\frac{a}{2 \tilde{A} + \beta}} \right]^{-3k -2-\alpha+j}\\
 C_{k,2} & = &  \frac{1}{2^{2k}} \sum_{j=k^*+1}^{4k} \tilde{c}_{4k,j} {\mathcal K}_{3k-j+(\alpha+2)}\bigg(\sqrt{4 a(2 \tilde{A} + \beta)}\bigg)
 \left[ \sqrt{\frac{a}{2 \tilde{A} + \beta}} \right]^{-3k -2-\alpha+j}
\end{eqnarray*}
where $k^*$ is such that
$3k-k^*+\alpha+2 > 0$
and $3k-k^*+\alpha+1 < 0.$ For $\alpha \in (-1,0),$
we have $3k-j+(\alpha+2)<0$ for 
$3k+2\leq j\leq 4k$ and $C_{k,2}$ includes the maximum number of terms,
that is
\begin{eqnarray}
C_{k,2} & = &  \frac{1}{2^{2k}} \sum_{j=3k+1}^{4k} \tilde{c}_{4k,j} {\mathcal K}_{j-3k-(\alpha+2)}\bigg(\sqrt{4 a(2 \tilde{A} + \beta)}\bigg)
 \left[ \sqrt{\frac{a}{2 \tilde{A} + \beta}} \right]^{-3k -2-\alpha+j} \nonumber\\
 & \sim & \frac{1}{2^{2k}} \sum_{j=3k+1}^{4k} \frac{\tilde{c}_{4k,j}}{\sqrt{j-3k-\alpha-2}} \left( \frac{j-3k-\alpha-2}{e (2 \tilde{A}+\beta)} \right)^{j-3k-\alpha-2}
 \label{Ck2}.
\end{eqnarray}
The  dominant term in \eqref{Ck2} is for
$j=4k$ and the asymptotic behaviour of $C_k$ is still of order $k^k$. 
\end{proof}
Observe that for $k=2,$ from \eqref{error1} we recover \eqref{errorPars}. For higher values of $n,$ a good choice is $k <\!\!\!\!\!< \sqrt{n}.$
\begin{rem}\label{limitcase}
Note that if 
$\beta = \frac{c_1[T]}{
    c_2[T]} = \frac{1}{\sigma^2} \left( \mu - \frac{\sigma^2}{2} \right)^2$
the integral $I_m$ in \eqref{IM} converges if and only if $\alpha > m-2.$ Indeed in such a case $2 \tilde{A} + \beta=0$ and $I_m$ in \eqref{IM} reduces to
\begin{equation}
I_{m} = \sum_{j=0}^{4m} \tilde{c}_{4k,j} \int_0^{\infty} \exp \left( - a y \right) \, y^{(3m+2+\alpha-j) -1} \,
 {\rm d} \,t.
 \label{specialcase}
\end{equation}
The integral on the rhs of \eqref{specialcase} is convergent if and only if $3m+2+\alpha-j>0$ for all $j=0,\ldots,4m,$ that is if and only if $\alpha > m-2.$ Therefore  \eqref{error1} still holds with $k < \alpha+2.$ In such a case we have 
\begin{eqnarray}
C_k & = & C I_k \propto \frac{1}{2^{2k}} \sum_{j=0}^{4k} \tilde{c}_{4k,j}
a^{3k + \alpha +2 -j} \Gamma(3k-j+\alpha+2)  \nonumber \\
& \propto &  \frac{1}{2^{2k}} \sum_{j=0}^{4k} \tilde{c}_{4k,j}
a^{3k + \alpha +2 -j} 
\frac{(3k-j)^{3k-j+\alpha+3/2}}{e^{3k-j}} \label{Ck2bis}
\end{eqnarray}
as $\Gamma(z+b) \sim \sqrt{2 \pi} e^{-z} z^{z+b-1/2}.$ As the leading term in \eqref{Ck2bis} is for $j=0,$ we still have $C_k=O(k^k).$ 
\end{rem}

Even if $\beta > \sigma^2 B^2,$ it is still possible to use the rhs of \eqref{expansion1} due to the following proposition. Since the result is a reformulation
of Theorem 2 in \cite{MR4159245}, the proof is omitted.
\begin{prop}\label{conv2}
For $t >0$ and $r \in (0,1),$ we have
$$U(\beta t, r):= 1 + \sum_{k \geq 1}  {\mathcal B}_k^{(\alpha)} L_k^{(\alpha)}(\beta t) r^k < \infty \quad \text{and }\quad \lim_{r \uparrow 1}U(\beta t, r) = 
\frac{g(S,t|y_0)}{f_{\alpha,\beta}(t)}.$$
\end{prop}
\subsection{\label{sec:level4}
Choosing a different reference pdf: the log-normal density}
Since the gamma pdf is such that  $f_{\alpha,\beta}(0) \ne 0$ for $\alpha  \in (-1,0),$ differently from the FPT pdf for which  $g(0)=0,$  we might test  a different reference pdf to recover the polynomial  approximation of $g(t).$
A density with support $(0,+\infty)$ and behaving as the FPT pdf of a GBM is the log-normal one with parameters $\mu$ and $\sigma$   
\begin{equation}
\tilde{f}_{\mu,\sigma}(t) = \frac{1}{t\sqrt{2\pi\sigma^2}}e^{-\frac{(\ln(t)-\mu)^2}{2\sigma^2}}.
\end{equation}
To recover a polynomial approximation, we need to characterize the
family of orthogonal polynomials with respect to the measure $\nu({\rm d}t) = \tilde{f}_{\mu,\sigma}(t) \, {\rm d}t.$
Unlike the generalized Laguerre ones, 
these polynomials are not classically known and have been computed for $\mu=0$ and $\sigma=1$ in \cite{chaos} and for arbitrary $\mu$ and $\sigma$ in \cite{asmussen, zheng}
using a classic  procedure (see for example \cite[Th. 2.1.1]{szeg},    \cite[Section 4]{prov1}). Using
the monic polynomials  given in \cite{zheng}, the polynomial approximation results to be 
\begin{equation}
\label{approxLOG}
\hat{g}_n(t) = C \; \tilde{f}_{\mu,\sigma}(t) \;\sum_{i=0}^{n}\eta_i\;\pi_i(t),\;\;n\in\mathbb{N}
\end{equation}  
where $C$ is a suitable normalization constant and
\begin{equation}
\pi_i(t) = 
\sum_{j=0}^{i} (-1)^{i+j}e^{(i-j)\mu}e^{(i-\frac{1}{2})(i-j)}
\gbinom{i}{j}_{e^{\sigma^2}} \; t^j   
\end{equation}
with $\gbinom{n}{i}_q$  the $q$-Binomial coefficient 
\begin{equation}
\gbinom{n}{i}_q=\frac{(1-q^n)(1-q^{n-1})\dotsm(1-q^{n-i+1})}{(1-q^i)(1-q^{i-1})\dotsm(1-q)}.
\end{equation}
The main drawback of \eqref{approxLOG} is that the system $\{\pi_i\}_{i\geq0}$ is not complete in $\mathcal{L}^2_{\nu}(0,\infty)$  \cite[
Proposition 1.1]{asmussen}. 
Indeed the log-normal distribution is not fully characterized by its moments \cite{heyde} and $\hat{g}_n$ might  converge to a density different from $g$, but sharing the same moments as $g$ (see
\cite[Proposition 4.1]{chaos} for a non trivial example of a family of densities for which the convergence fails). Note that 
\eqref{approxLOG} fails to approximate even a log-normal pdf \cite[Fig 1.2]{asmussen}. 
\subsection{\label{sec:level4bis}
Choosing a different reference pdf: the Inverse Gaussian density}

Another possible choice for the reference density could be the Inverse Gaussian. In the special cases of the  GBM this choice would be clearly extremely convenient since we actually know that the FPT has Inverse Gaussian distribution.
In this case the choice may seem nearly cheating, but it is also the FPT density of the Brownian motion and it would make sense to use it  as reference density.
Unfortunately in \cite{nishii1996orthogonal} it is shown that
the usual method of differentiating the density does not lead to an orthogonal polynomial system, and starting
from the Laguerre polynomials leads to a system of orthogonal functions which is not complete.
The only way to get a complete system of polynomials is by using
the Gram–Schmidt orthogonalisation procedure, but the resulting polynomials are not easy to use (see also \cite{laub}).
In \cite{hassairi2004characterization} the authors propose a method
to derive the polynomials that involves the so-called bi-orthogonality property but they do not discuss whether this construction leads to a basis.

Therefore, in the following we have considered only the gamma pdf as reference density.

\section{\label{section4}
Computational issues }
In the software environment {\tt R}, the package {\tt PDQutils} contains a collection of tools for approximating pdf's via classical expansions involving moments and cumulants.
For the Laguerre polynomials, the {\tt PDQutils} routine  implements
the following choice of $\alpha$ and $\beta:$
\begin{equation}
\label{condsuaeb}
\alpha := \frac{c_1^2[T]}{c_2[T]}-1 \quad \text{and }\quad \beta:= \frac{c_1[T]}{c_2[T]}.
\end{equation}
In such a case, expansion \eqref{expansion2} simplifies as a straightforward computation shows that ${\mathcal B}_1^{(\alpha)}=
 {\mathcal B}_2^{(\alpha)}=0$ and the first two moments of $T$ are equal to  the first two moments of $f_{\alpha,\beta}(t),$ that is  
\begin{equation}
\label{mm}
\mathbb E[T] = \frac{\alpha+1}{\beta} 
 \quad \text{and }\quad 
\mathbb E[T^2] = \frac{(\alpha+1)(\alpha+2)}{\beta^2}.
\end{equation}
If the {\tt PDQutils} routine is used within an iterative procedure aiming to return the integer $n$ that allows a good approximation, such a procedure results computationally inefficient since the $n+1$-th approximation $\hat{g}_{n+1}(t)$ cannot be obtained by updating the $n$-th one $\hat{g}_{n}(t)$. 

Here, we propose a different approach that relies on nested products taking advantage of a different representation of the polynomial $p_n(t)$
in \eqref{expansion2}. Indeed, combining the coefficients of $\{L_k^{\alpha}(\beta t)\}$ with the same power of $t,$ the terms of the polynomial $p_n(t)$ in \eqref{expansion2} can be rearranged as follows 
\begin{equation}
\label{expansion3}
p_n(t) = \sum_{k=0}^n h_{n,k} \frac{(-\beta t)^k}{k!} 
\quad \text{with}
\quad 
h_{n,k} = \sum_{j=k}^n {\mathcal B}_j^{(\alpha)}
\binom{\alpha+j}{j-k},
\end{equation}
${\mathcal B}_0^{(\alpha)} = 1$ and
\begin{equation}
\label{hnk}
\binom{\alpha+j}{j-k} =
\left\{ \begin{array}{ll}
1, & j=k,\\
\frac{(\alpha+j)(\alpha+j-1)\cdots (\alpha+k+1)}{(j-k)!},
& j > k.
\end{array}\right.
\end{equation}
Therefore $p_n(t)$ is better evaluated using the recurrence relation
\begin{equation}
\label{recursion}
d_{n,i}(t)=h_{n,i-1}-\frac{\beta t}{i} d_{n,i+1}(t)  \quad
\text{for } i=n,n-1,\ldots,1
\end{equation}
with the initial condition
$d_{n,n+1}(t) = h_{n,n},$
since the last value gives $d_{n,1}(t)=p_n(t).$ 
Also the coefficients $\{{\mathcal B}_{k}^{(\alpha)}\}$ can be computed using a recursion formula, as the following proposition shows.
\begin{prop} For all $k \geq 1$ we have
\begin{equation}
{\mathcal B}_{k}^{(\alpha)} =
\sum_{j=1}^k \binom{k}{j} (-1)^{j+1} {\mathcal B}_{k-j}^{(\alpha)} + \frac{(-\beta)^k {\mathbb E}[T^k]}{(\alpha+k)_k}.
\label{recursionB}
\end{equation}
\end{prop}
\begin{proof}
By plugging \eqref{expansion1} into ${\mathcal B}_{k-j}^{(\alpha)}$ for
$j=1,\ldots,k$ we get
\begin{equation}
\label{formula1}
\sum_{j=0}^k \binom{k}{j} (-1)^j {\mathcal B}_{k-j}^{(\alpha)} = 
{\mathcal B}_{k}^{(\alpha)} + \sum_{j=1}^k \binom{k}{j}
(-1)^j \bigg[ 1 + \sum_{i=1}^{k-j} \binom{k-j}{i} \frac{(-\beta)^i {\mathbb E}[T^i]}{(\alpha+i)_i} \bigg].
\end{equation}
Moreover, by expanding the inner sum in the rhs of  \eqref{formula1} and grouping with respect to the $j$-th moment ${\mathbb E}[T^j]$ we have
$$
\sum_{j=1}^k \binom{k}{j}
(-1)^j \bigg[ \sum_{i=1}^{k-j} \binom{k-j}{i} \frac{(-\beta)^i {\mathbb E}[T^i]}{(\alpha+i)_i} \bigg] = \sum_{j=1}^{k-1}
\frac{(-\beta)^j {\mathbb E}[T^j]}{(\alpha+j)_j} 
\bigg[ \sum_{i=1}^{k-j}
(-1)^i \binom{k}{i} \binom{k-i}{j} \bigg].
$$
Since 
$$\sum_{i=1}^{k-j}
(-1)^i \binom{k}{i} \binom{k-i}{j} =  \binom{k}{j} \sum_{i=1}^{k-j} \binom{k-j}{i}(-1)^i=-\binom{k}{j}$$
and 
$\sum_{j=0}^k \binom{k}{j}
(-1)^j=0$ which gives $\sum_{j=1}^k \binom{k}{j}(-1)^j=-1,$
from \eqref{formula1} we get
\begin{equation}
\sum_{j=0}^k \binom{k}{j} (-1)^j {\mathcal B}_{k-j}^{(\alpha)} = 
{\mathcal B}_{k}^{(\alpha)} -1 -\sum_{j=1}^{k-1}\binom{k}{j}
\frac{(-\beta)^j {\mathbb E}[T^j]}{(\alpha+j)_j}. 
\end{equation}
Plugging \eqref{expansion1}
into ${\mathcal B}_{k}^{(\alpha)}$ after some algebraic manipulation, we get
\begin{equation}
\frac{(-\beta)^k {\mathbb E}[T^k]}{(\alpha+k)_k} =
\sum_{j=0}^k \binom{k}{j} (-1)^j {\mathcal B}_{k-j}^{(\alpha)}.
\label{formula2}
\end{equation}
from which \eqref{recursionB} follows.
\end{proof}
For the GBM FPT random variable  also the moments $\{{\mathbb E}[T^k]\}$ of the FPT random variable can be computed through recursion using 
\eqref{recursionmom}.
For different stochastic processes, if cumulants $\{c_{k}[T]\}$ are known \cite{MR4159245,mathematics2021}, the following recursion might be 
implemented \cite{DiNardo}:
\begin{equation}
{\mathbb E}[T^{n+1}]= c_{n+1}[T] + \sum_{k=1}^{n}\binom{n}{k-1} c_k[T] 
\, {\mathbb E}\big[T^{n+1-k}\big].
\label{recmom}
\end{equation}
Therefore, the updating of   $\hat{g}_n(t)$ to $\hat{g}_{n+1}(t)$ might be performed by using the recursion \eqref{recursion}, setting  $h_{n+1,n+1}={\mathcal B}_{n+1}^{(\alpha)}$ and updating the coefficients $h_{n,i}$ to $$h_{n+1,i}=h_{n,i} + {\mathcal B}_{n+1}^{(\alpha)} \binom{\alpha+n+1}{n+1-i} \quad \text{for } i=0, \ldots, n.$$ 
It must be noted that we cannot in practice take $n$ arbitrarily large, due to
numerical errors incurred in calculating the  coefficients $\{{\mathcal B}_k^{(\alpha)}\}$. Obviously, this can be overcome by using infinite precision operations. Software tools like {\tt Mathematica} allow for arbitrarily large but finite precision. However this swiftly becomes prohibitively slow.
A way to push the iteration procedure up to the best order of numerical approximation relies on the subsequent normalization condition satisfied by the sequence $\{h_{n,i}\}.$ 
\begin{prop}
\label{stopcoeff}
For all $n \geq 0$ we have 
\begin{equation}
h_{n,0}+\sum_{i=1}^n \frac{(-1)^i}{i!} h_{n,i} (\alpha +i)_i = 1.
\label{cond1}
\end{equation}
\end{prop}
\begin{proof}
From the normalization condition of
$\hat{g}_n(t),$  
we have
$$\sum_{i=0}^n \frac{(-1)^i}{i!} \beta^i \, h_{n,i}  \int_0^{\infty} t^i f_{\alpha,\beta}(t)  {\rm d} t = 1 \quad \text{for all } \; n \geq 0.
$$
The result follows
by observing that the integrals in the lhs of \eqref{cond1} are the moments of the gamma pdf $f_{\alpha,\beta}(t),$ that is $$\int_0^{\infty} t^i f_{\alpha,\beta}(t)  {\rm d} t = \frac{\Gamma(\alpha+1+i)}{\beta^i \Gamma(\alpha+1)}.$$
\end{proof}
Condition \eqref{cond1} has been used to test the numerical stability of the computations, that is the iteration is ended as soon as this condition is no longer verified.

\section{\label{section5}
Numerical results }

In this section the Laguerre-Gamma polynomial
approximation  \eqref{expansion2} discussed previously is applied to the FPT pdf of the GBM. Since we have the closed form \eqref{FPTpdf}, we will
be able to compare directly the approximation with the true density.
We choose to employ $p_n(t)$ in  \eqref{expansion2} in the form \eqref{expansion3}, that gives the advantage of a more efficient implementation.
Moments are calculated using recursion \eqref{recursionmom}.

As previously stated, the main issue in the approximation is the choice of the best degree $n$ of the polynomial  approximation.  Different possibilities arise. For example, we have considered using a convergence-based stopping criterion. It consists in choosing the smallest $n$ such that, for a fixed tolerance $\epsilon>0$, we have
\begin{equation}\label{criterion1}
\Gnorm{\hat{g}_n(t)-\hat{g}_{n-1}(t)}_{\alpha,\beta}<\epsilon.
\end{equation}
However, this criterion is affected by numerical instability, because the condition \eqref{criterion1} is satisfied  for large values of $n$ but the corresponding  numerical amount of errors has already compromised the approximation. Consequently, we choose to employ the stopping criterion \eqref{cond1} as follows.  Set 
\begin{equation}
\hat{h}_n = h_{n,0}+\sum_{i=1}^{n}\frac{(-1)^i}{i!}h_{n,i}(\alpha+1)_i.
\end{equation}
Since $\hat{h}_n=1$ from Proposition  \ref{stopcoeff}, as stopping criterion  we choose the smallest $n$ such that 
	 
\begin{equation}\label{criterion2}
|\hat{h}_{n+1}-1|>\epsilon, \quad \mbox{ for a fixed } \  \epsilon>0.
\end{equation}
For the parameters $\alpha$ and $\beta,$ according to  \eqref{condsuaeb} we have
  	\begin{equation}
 	\beta=\frac{1}{\sigma^2}\left(\mu-\frac{\sigma^2}{2} \right)^2
 	\;\;\;\text{and}\;\;\;
 	\alpha=\frac{1}{\sigma^2}\log\left(\frac{S}{y_0}\right)
 	\left(\mu-\frac{\sigma^2}{2}\right)-1.
 	\end{equation}
Note that with these choices, we are in the limit case discussed in Remark \ref{limitcase} for what concerns the error.

To analyse the efficiency and the usefulness of the proposed method, in the following  we consider two instances:
\begin{enumerate}
    \item  the FPT pdf $g$ has  moments or cumulants known in a closed form;
    \item the knowledge of the FPT moments or cumulants is limited to first few orders.
\end{enumerate}
Indeed, for most of the stochastic processes the knowledge of the moments is limited to the mean and the variance or the expression of higher moments is cumbersome and not computationally convenient  to be employed in recovering $\hat{g}_n,$ see for example \cite{mathematics2021}.  In such cases the approximation might be carried out  by simulating  the trajectories of the process  through a suitable  Monte Carlo method and estimating the moments/cumulants of $T$.  In the last paragraph we suppose $g$ not known and use the approximation $\hat{g}_n$ to perform parameter estimations
using the maximum likelihood method. 
	 
\subsection{Laguerre-Gamma polynomial approximation:
known moments}
 	In the following we will test the accuracy and efficiency of the approximation by comparing $\hat{g}_n$ in \eqref{expansion2} with the true FPT pdf $g$ in \eqref{FPTpdf} through the following two approaches:
 	\begin{enumerate}[(a)]
 		\item a graphical approach by comparing their plots,
 		\item a quantitative approach by computing $|g(t)-\hat{g}_n(t)|$.
 	\end{enumerate}
 	We show and study three different cases (A,B and C in Fig. \ref{fig_2}), where we have fixed $S=10$ and $y_0=1$.
 	\begin{figure}[ht]
 		\centerline{\includegraphics[scale=0.67]{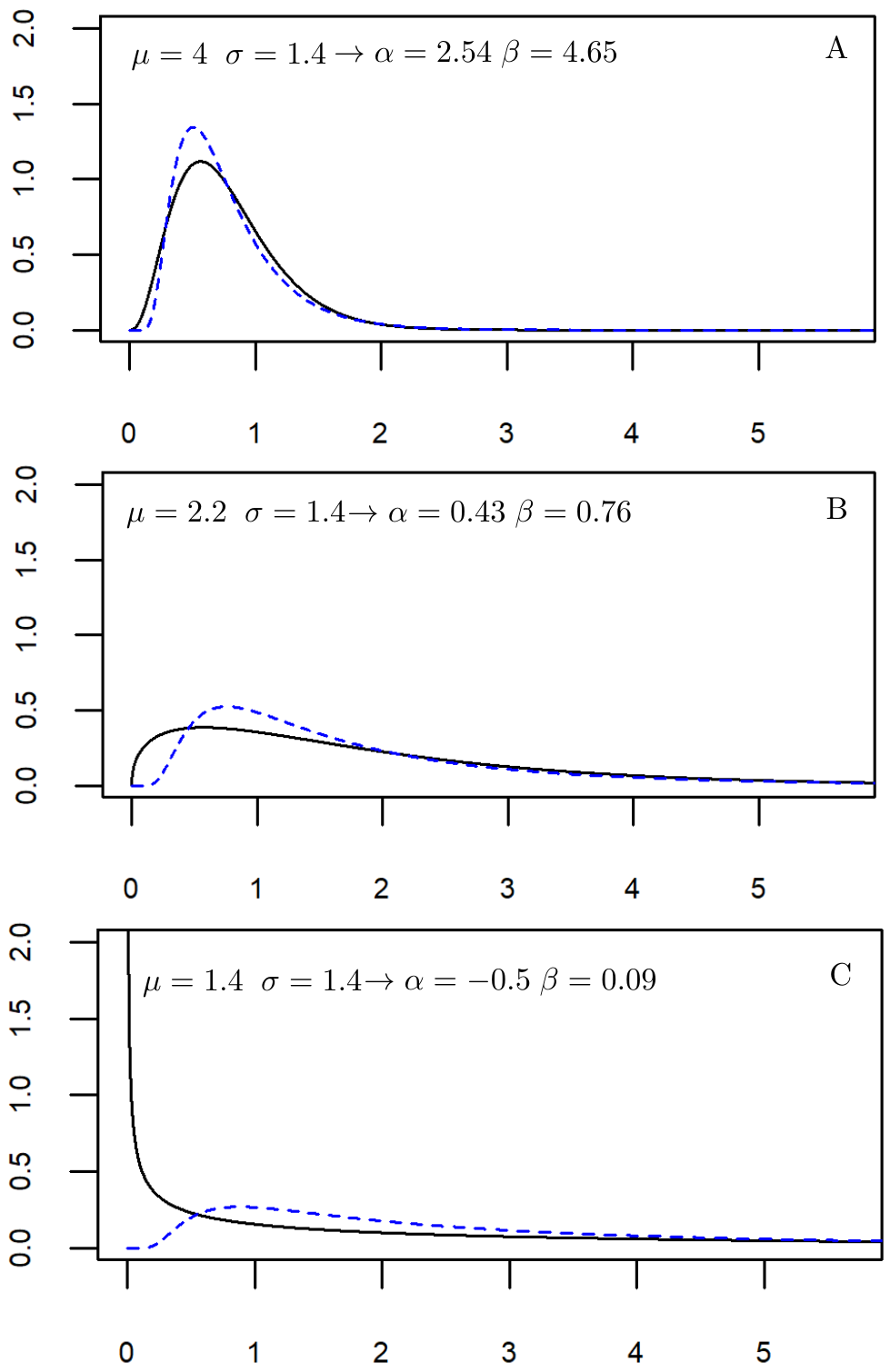}}
 		\caption{Plots of the FPT pdf of the GBM 
 		(blue dashed line) and the corresponding reference pdf $f_{\alpha,\beta}$ (black solid line), for $S=10$ and $y_0=1.$ The other values of the parameters are given in the legends. }
 		\label{fig_2}
 	\end{figure}
  	These examples show that, as the parameters change, the FPT pdf $g$ and the reference pdf $f_{\alpha,\beta}$ can be significantly different.
 	In Figs. \ref{Approx_1}, \ref{Approx_2} and \ref{Approx_3} we have plotted the polynomial approximation $\hat{g}_n(t)$ (black solid line) and the true density $g(t)$ (blue dashed line)  for the three mentioned cases using  four different orders of approximation  each time.
 	
 \begin{figure}[ht]
	\centerline{\includegraphics[scale=1]{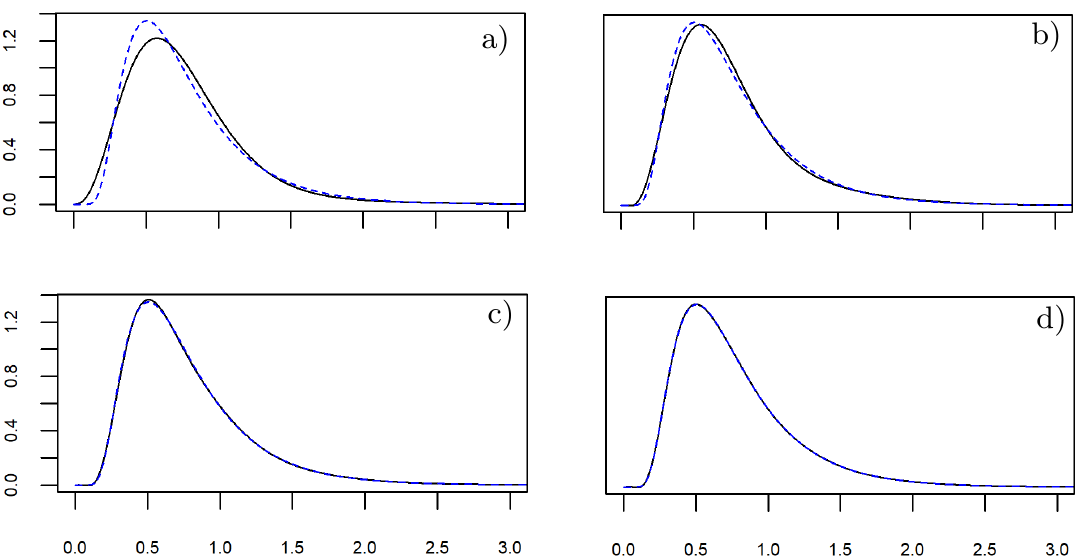}}
	\caption{Plots of the polynomial approximation $\hat{g}_n$ (black solid line) and of the true density $g$ (blue dashed line) in case A with $S=10$, $y_0=1$, $\mu=4$ and $\sigma=1.4$ for $n=3$ in {\it a)} $n=5$ in {\it b)} $n=16$ in {\it c)} 
	and $n=30$ in {\it d)}, where the last $n$ is the minimum integer s.t. condition \eqref{criterion2} is satisfied.}
	\label{Approx_1}
\end{figure}
\begin{figure}[ht]
	\centerline{\includegraphics[scale=1.1]{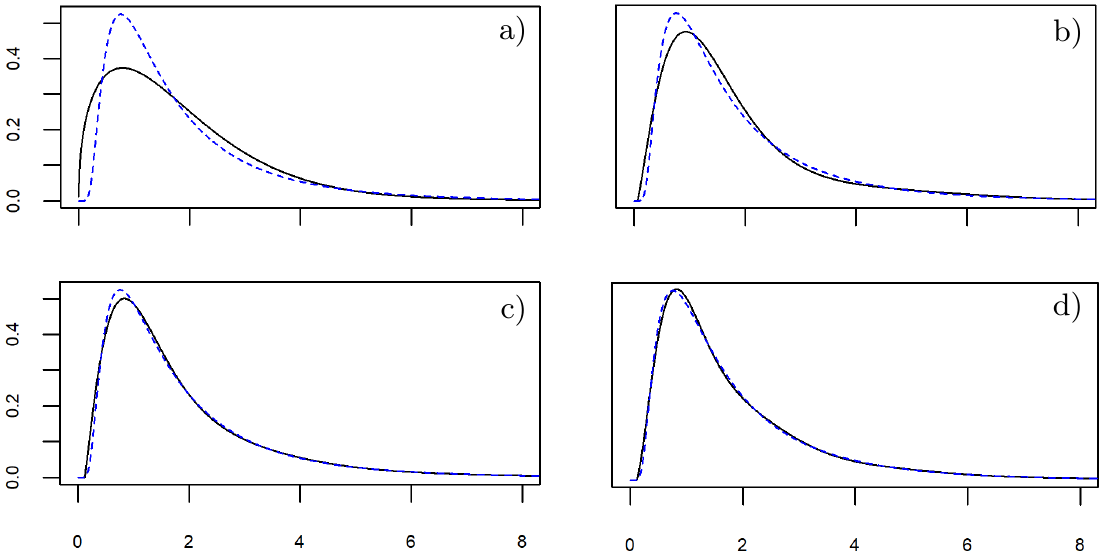}}
	\caption{Plots of the polynomial approximation $\hat{g}_n$ (black solid line) and of the true density $g$ (blue dashed line) in case B with $S=10$, $y_0=1$, $\mu=2.2$ and $\sigma=1.4$ for $n=3$ in {\it a)} $n=8$ in {\it b)} $n=16$ in {\it c)} 
	and $n=29$ in {\it d)}, where the last $n$ is the minimum integer s.t. condition \eqref{criterion2} is satisfied.}
	\label{Approx_2}
\end{figure}
\begin{figure}[ht]
	\centerline{\includegraphics[scale=1.1]{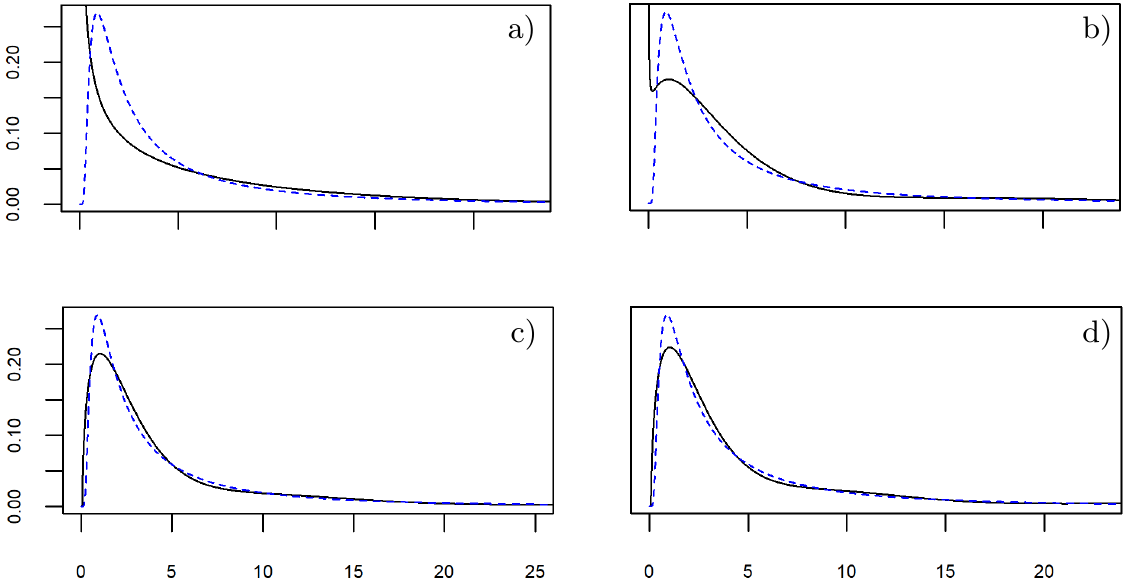}}
	\caption{Plots of the polynomial approximation $\hat{g}_n$ (black solid line) and of the true density $g$ (blue dashed line) in case C with $S=10$, $y_0=1$, $\mu=1.4$ and $\sigma=1.4$ for $n=3$ in {\it a)} $n=15$ in {\it b)} $n=25$ in {\it c)} 
	and $n=36$ in {\it d)}, where  the last $n$ is the minimum integer s.t. condition \eqref{criterion2} is satisfied.}
	\label{Approx_3}
\end{figure}
 	By comparing the three figures we observe that:
 	\begin{enumerate}[(a)]
 		\item in case A, where the reference pdf $f_{\alpha,\beta}$ and the FPT pdf $g$ have a similar behavior, a low degree $n$ guarantees a good approximation,
 		\item as $\alpha$ decreases, the reference pdf loses its typical bell shape and deviates further away from the FPT pdf,
 		\item in all the considered cases the goodness of the approximation increases with $n$ as long  as \eqref{cond1} is satisfied,
 		\item compared with the cases A and B, the case C proves to be the most difficult; in fact even for $n=36$ we do not match the peak of the FPT pdf $g$. Indeed, according to Remark \ref{limitcase}, in this case the order of convergence is $1/n$ as Theorem \ref{3.3} does not hold. So to get a good approximation we should consider high values of $n$ but this is prevented by the increasing numerical errors, namely condition \eqref{criterion2} is reached quickly. Moreover for $\alpha<0$, the mode of the reference pdf is not defined and thus the initial approximation is very far from $g$. In addition for the choice of parameters $\sigma^2$ and $\mu$  the stochastic component of the dynamics of the process is more dominant than the deterministic one, resulting in a flatter FPT pdf. 
 	\end{enumerate}    
 	
 	In Fig. \ref{Error} we have plotted the absolute error $|g(t)-\hat{g}_n(t)|$ between the true FPT pdf $g$ and the approximated $\hat{g}_n$ for the three cases and for the smallest $n$ such that condition \eqref{criterion2} is satisfied.
 	\begin{figure}
	\centerline{\includegraphics[scale=1]{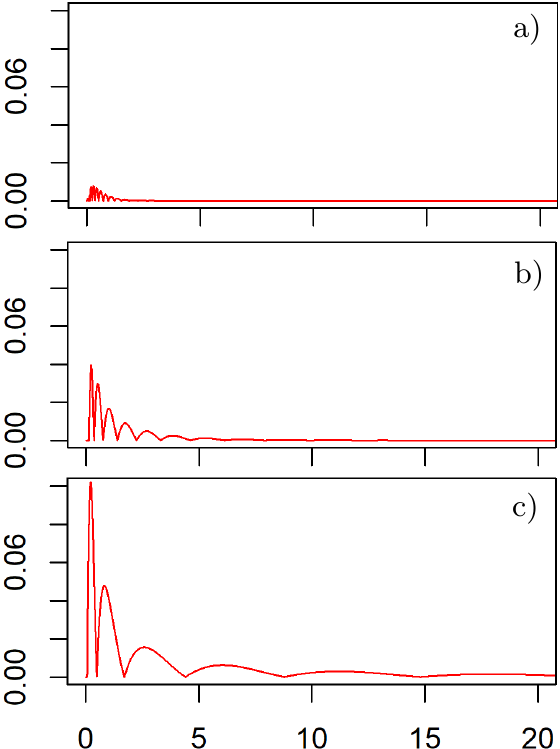}}
	\caption{Plots of the absolute error $|g(t)-\hat{g}_n(t)|$ between the true FPT pdf $g$ and the approximation $\hat{g}_n$ for the  smallest $n$ s.t. condition \eqref{criterion2} is satisfied in case A with $S=10$, $y=1$, $\mu=4$ and $\sigma=1.4$ in {\it a)}, case B with $S=10$, $y_0=1$, $\mu=2.2$ and $\sigma=1.4$ in {\it b)} and case C with $S=10$, $y=1$, $\mu=1.4$ and $\sigma=1.4$ in {\it c)}.}
	\label{Error}
\end{figure}

 \subsection{Laguerre-Gamma polynomial approximation:
unknown moments}
In the previous paragraph, the approximation relies on the knowledge of the moments of the underlying process, which were used in the coefficients $\{{\mathcal B}_{k}^{(\alpha)}\}$ defined in \eqref{expansion1} and in the approximation \eqref{expansion2}.
In the following we will consider a different setup. 

We sample $10^4$ FPTs of the GBM using the Milstein Method \cite{Higham} to simulate the trajectories of the process, thus obtaining a  random i.i.d. sample of FPTs $\{T_1,\dots,T_{N}\}$ of size $N=10^4$. Then the approximation \eqref{expansion2} is carried out using the recursion \eqref{recmom} and $\kappa$-statistics in place of cumulants $\{c_k(T)\}.$
The $k$-th $\kappa$-statistic is a symmetric function of the random sample whose expectation gives the $k$-th cumulant
$c_k(T).$ We have used the {\tt R}-package {\tt kStatistics} \cite{kstat} to recover $\kappa$-statistics for the simulated sample of FPTs. The numerical results 
 show that the approximations behave in essentially the same way as in the case of known moments and so we do not find useful to show the respective plots. 

Instead, in Fig. \ref{Error_s} we have plotted the absolute error $|g(t)-\hat{g}_n(t)|$ between the true FPT pdf $g$ and the approximated $\hat{g}_n$ for the three cases of Fig.\ref{fig_2} and for the smallest $n$ such that condition \eqref{criterion2} is satisfied. We note that the absolute errors are larger (in particular case A) but this strategy is much more general and can be applied to any process whose trajectories can be simulated or to FPT data with unknown underlying process.

\begin{figure}
	\centerline{\includegraphics[scale=1]{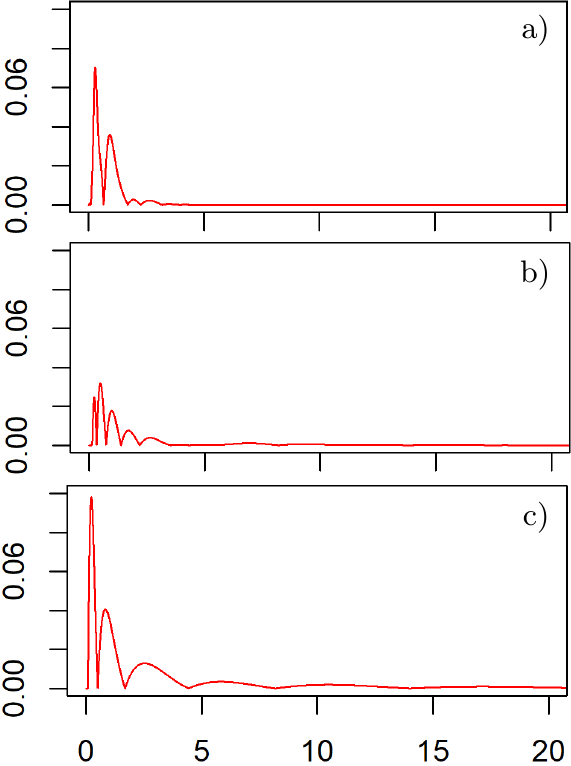}}
	\caption{Plots of the absolute error $|g(t)-\hat{g}_n(t)|$ between the true FPT pdf $g$ and the approximation $\hat{g}_n$ using sample moments for the  smallest $n$ s.t. condition \eqref{criterion2} is satisfied in case A with $S=10$, $y=1$, $\mu=4$ and $\sigma=1.4$ in {\it a)}, case B with $S=10$, $y=1$, $\mu=2.2$ and $\sigma=1.4$ in {\it b)} and case C with $S=10$, $y_0=1$, $\mu=1.4$ and $\sigma=1.4$ in {\it c)}.}
	\label{Error_s}
\end{figure}

 \subsection{Parameters estimation} 
 If the true pdf $g$ were known, the well known method of maximum likelihood estimation could have been applied to carry out parameter estimations of the process. In this paragraph we analyze the possibility to use the approximated density $\hat{g}_n$ with known moments, instead of the true density $g$ in the likelihood function. This approach would be particularly useful if the true density $g$ is not known, as usual happens, and can be applied since $\hat{g}$ is a pdf. Suppose to have  a  sample of FPTs $\{T_1,\dots,T_{N}\}.$ 
The maximum likelihood estimate of $(\mu,\sigma^2)$ is
\begin{equation}\label{mleeq}
    (\hat{\mu},\hat{\sigma})=\argmax_{(\mu,\sigma^2)\in\Theta}\ell_N(\mu,\sigma^2),
\end{equation}
with $\Theta=(-\infty,+\infty)\times(0,+\infty),$  $\ell_N(\mu,\sigma^2)=\ln{L_N(\mu,\sigma^2)}$ the log-likelihood function and 
$$L_N(\mu,\sigma^2)= \prod_{i=1}^{N} \hat{g}_n(T_i;\mu,\sigma^2).$$.
 
The maximization problem \eqref{mleeq} has been numerically solved using 
the function {\tt Optim} 
in the base {\tt{R}}-package {\tt Stats}. More specifically the maximum likelihood
estimates of the parameters have been  obtained through a global optimization algorithm known as {\sl Simulated Annealing}  \cite{anneal}. Simulated annealing is a stochastic global optimisation technique applicable to a wide range of discrete and continuous variable problems. It makes use of Markov Chain Monte Carlo samplers, to provide a means to escape local optima by allowing moves which worsen the objective function, with the aim of finding a global optimum. Technical details can be found in \cite{anneal}, a variant of which is the algorithm implemented in {\tt Optim}.  
For the same sample $\{T_1,\dots,T_{N}\}$ used in  the previous subsection  with $N=10^4$, Table \ref{mletab}  shows the maximum likelihood estimates of $\mu$ and $\sigma^2$ in\eqref{mleeq} for $n=34$ and for the cases considered in Fig.\ref{fig_2}.

\begin{table}[h]
	\begin{center}
		\begin{tabular}{| c | l | l | l | l |}
			\hline 
		 & $\mu$ & $\hat{\mu}$ & $\sigma^2$ & $\hat{\sigma}^2$ \\ \hline
			A & 4 & 3.893 & $(1.4)^2$ & $(1.362)^2$ \\ \hline 
			B & 2.2 & 2.152 & $(1.4)^2$ & $(1.37)^2$ \\ \hline
			C & 1.4 & 1.18 & $(1.4)^2$ & $(1.24)^2$ \\ \hline
		\end{tabular}
	\end{center}
	\caption{The true parameters $\mu$ and $\sigma^2$ and the
	maximum likelihood estimates $\hat{\mu}$ and $\hat{\sigma}^2$ for the cases A,B and C of Fig.\ref{fig_2}.}
	\label{mletab}
\end{table}
\begin{rem}
In the case of GBM process, a simpler way to estimate the parameters is using the method of moments, since the cumulants of the $IG(a,b)$ pdf have the simple expression \eqref{cumulants}. To estimate the cumulants we employ the $\kappa$-statistics $\kappa_i.$
Starting from the 
sample $\{T_1,\dots,T_N\}$ with $N=10^4$ used in the previous subsection, we compute the $\kappa$-statistics $\hat{\kappa_1}$ and $\hat{\kappa_2}$ using the {\tt R}-package {\tt kStatistics} \cite{kstat}. 
By simple computations, we set $\hat{\kappa_1}=b$ and $\hat{\kappa_2}=\frac{b^2}{a}\hat{\kappa_1}$ and
 we obtain  the following closed-form expressions for the estimations of $\mu$ and $\sigma^2$ 
 \begin{equation}\label{mmeq}
 \hat{\mu}=\frac{S_{y_0}}{\hat{\kappa_1}}\left(1+\frac{1}{2}\frac{\hat{\kappa_2}}{(\hat{\kappa_1})^2}S_{y_0}\right)\;\;\;\;\text{and}\;\;\;\;\hat{\sigma}^2=\frac{\hat{\kappa_2}}{(\hat{\kappa_1})^3}S_{y_0}, \end{equation}
 where $S_{y_0}=\ln(S)-\ln(y_0)>0$.
 Table \ref{mmtab} shows the numerical estimations of $\mu$ and $\sigma^2$ for the cases considered in Fig.\ref{fig_2}.

 \begin{table}[h]
	\begin{center}
		\begin{tabular}{| c | l | l | l | l |}
			\hline 
			 & $\mu$ & $\hat{\mu}$ & $\sigma^2$ & $\hat{\sigma}^2$ \\ \hline
			A & 4 & 3.899 & $(1.4)^2$ & $(1.39)^2$\\ \hline 
			B & 2.2 & 2.203 & $(1.4)^2$ & $(1.392)^2$ \\ \hline
			C & 1.4 & 1.39 & $(1.4)^2$ & $(1.39)^2$ \\ \hline
		\end{tabular}
	\end{center}
	\caption{The true parameters $\mu$ and $\sigma^2$ and the estimated parameters $\hat{\mu}$ and $\hat{\sigma}^2$ using Eqs. \eqref{mmeq} for the cases A,B and C of Fig.\ref{fig_2}}
	\label{mmtab}
\end{table}
Although the estimations in Table \ref{mmtab} are more accurate than those obtained in Table \ref{mletab}, we stress that in general the closed form expressions of the cumulants are not so manageable or even unavailable. In such a case, the maximum likelihood estimation using $k$-statistics is the only possible strategy.
\end{rem}  

\bibliographystyle{abbrv}
\bibliography{refs}

\end{document}